\documentclass[11pt]{amsart}

\usepackage{amssymb,amsmath,amsthm,amscd}
\usepackage[svgnames]{xcolor}
\usepackage{hyperref}
\hypersetup{colorlinks,%
    citecolor=Purple,%
    linkcolor=Blue,%
}

\advance\oddsidemargin by -1cm
\advance\evensidemargin by -1cm 
\newtheorem{theorem}{Theorem}

\newtheorem{proposition}{Proposition}

\newtheorem{lemma}{Lemma}

\theoremstyle{definition}

\newcommand{\bdm}{\begin{displaymath}}
\newcommand{\edm}{\end{displaymath}}
\newcommand{\bq}{\begin{equation}}
\newcommand{\eq}{\end{equation}}
\newcommand{\bqn}{\begin{equation*}}
\newcommand{\eqn}{\end{equation*}}

\newcommand{\N}{{\mathbb N}}

\newcommand{\C}{{\mathbb C}}
\newcommand{\R}{{\mathbb R}}

\renewcommand{\epsilon}{\varepsilon}
\renewcommand{\phi}{\varphi}
\renewcommand{\rho}{\varrho}

\newcommand{\norm}[1]{\left\lVert#1\right\rVert}
\newcommand{\g}{{\bf \mathfrak g}}
\renewcommand{\k}{{\bf \mathfrak k}}
\renewcommand{\a}{{\bf\mathfrak a}}

\newcommand{\n}{{\bf\mathfrak n}}
\newcommand{\p}{{\bf \mathfrak p}}

\newcommand{\Hom}{\rm{Hom}}
\newcommand{\Ad}{\mathrm{Ad}\,}
\newcommand{\ad}{\mathrm{ad}\,}

\begin{document}

\author{Gang Liu, Aprameyan Parthasarathy}
\title[Domains of holomorphy]{Domains of holomorphy for irreducible admissible uniformly bounded Banach representations of simple Lie groups}
\address{Gang Liu, Institut {\'E}lie Cartan de Lorraine, Universit{\'e} de Lorraine, Ile du Saulcy,
57045 Metz, France.}
\email{gang.liu@univ-lorraine.fr}
\address{ Aprameyan Parthasarathy, Institut f\"ur Mathematik, Universit\"at Paderborn, Warburger Stra{\ss}e 100, 33098 Paderborn, Germany.}
\email{apram@math.upb.de}
\subjclass{}
\keywords{}
\thanks{The authors would like to thank Bernhard Kr\"otz for introducing the subject to them, and for helpful correspondence.}

\begin{abstract}
In this note, we address a question raised by  Kr\"otz on the classification of domains of holomorphy of irreducible admissible Banach representations for connected non-compact simple real Lie groups $G$. When $G$ is not of Hermitian type, we give a complete description of the domains of holomorphy for irreducible admissible \emph{uniformly bounded} representations on uniformly convex uniformly smooth Banach spaces and, in particular, for all irreducible uniformly bounded Hilbert representations. When the group $G$ is Hermitian, we determine the domains of holomorphy only when the representations considered are highest or lowest weight representations.
\end{abstract}

\maketitle

\section{Introduction}
 Let $G$ be a simple, non-compact, connected real Lie group with Lie algebra $\g$ and let $K$ be a maximal compact subgroup of $G$. Further, let $G_{\C}$ be the universal complexification of $G$ and $K_{\C}$, that of $K$. We can assume, without loss of generality, that $G \subseteq G_{\C}$ and that $G_{\C}$ is simply connected (see \cite[Remark $5.2$]{kroetz08}). Let $(\pi,V)$ be a Banach representation of $G$, i.e. assume there is a continuous action 
 \bqn
 G \times V \longrightarrow V, \qquad (g,v) \mapsto g\cdot v, \qquad g \in G, v \in V
 \eqn
of $G$ on a Banach space $V$ which gives rise to a group homomorphism $g \mapsto \pi(g)$ with $\pi(g)v:= g\cdot v$. For much of this introductory material \cite[Chapters $1$, $3$]{wallach1}) is a good reference. 
We call a vector $v \in V$ an \emph{analytic vector} if the \emph{orbit map} $\gamma_v: G \longrightarrow V$ of $v$, given by $g\longmapsto \pi(g)v$ and a priori continuous, extends to a holomorphic ($V$-valued) function on an open neighbourhood of $G$ in $G_{\C}$ or equivalently, if $\gamma_v$ is a ($V$-valued) real analytic map. Note that the space $V^{\omega}$ of analytic vectors for the representation $(\pi,V)$ is a $G$-invariant subspace which is dense in $V$. 
Recall that a Banach representation $(\pi, V)$ is called \emph{admissible} if $\dim \Hom_K(W,V_K)< \infty$ for any finite-dimensional $K$-module $W$, where $V_K$ is the space of $K$-finite vectors of $(\pi, V)$. If $\pi$ is irreducible and admissible, then we know that $K$-finite vectors are necessarily analytic vectors i.e., $V_K \subseteq V^{\omega}$.  So, given a non-zero vector $v\in V_K$, one might ask: to which natural domain in $G_{\C}$ does its orbit map $\gamma_v$ extend holomorphically. A first remark is that it is not unreasonable to expect that such a domain  would be independent of the vector $v \in V_K$ because $\mathcal{U}(\g_{\C})\cdot v=V_K$ as \,$\mathcal{U}(\g_{\C})$-modules. Here $\mathcal U(\g_{\C})$ is the universal enveloping algebra of $\g$. 
In fact, in \cite[Theorem 5.1]{kroetz08} Kr\"otz proved a classification of such domains when $\pi$ is an irreducible unitary representation and, further, proposed the following generalisation. 
 
\emph{Conjecture}: 
\label{thm: main theorem}
 Let $(\pi, V)$ be an irreducible admissible Banach representation of $G$. Given $0 \neq v \in V_K$,  there exists a unique maximal $G\times K_{\C}$-invariant domain $D_{\pi} \subset G_{\C}$ such that the orbit map $\gamma_v: g \mapsto \pi(g)v$ extends to a holomorphic map $\widetilde{\gamma}_v: D_{\pi} 
\longrightarrow V$. In more detail, we have
 \begin{itemize}
 \item[i)] $D_{\pi}=G_{\C}$ if $\pi$ is finite-dimensional.
 \item[ii)] $D_{\pi}=\widetilde{\Xi}^+$ if $G$ is Hermitian, and $\pi$ is a highest weight representation.
 \item[iii)] $D_{\pi}=\widetilde{\Xi}^-$ if $G$ is Hermitian, and $\pi$ is a lowest weight representation.
 \item[iv)] $D_{\pi}=\widetilde{\Xi}$\;\; in all other cases.
 \end{itemize}

Here we have written $\widetilde{\Xi}=q^{-1}(\Xi)$, $\widetilde{\Xi}^{\pm}=q^{-1}(\Xi^{\pm})$, where $q:G_{\C} \longrightarrow G_{\C}/K_{\C}$ is the canonical projection, $\Xi$ is the so-called \emph{crown domain} and $\Xi^+$, $\Xi^-$ are related domains in $\mathbb{X}_{\C}=G_{\C}/K_{\C}$ containing the Riemannian symmetric space $\mathbb{X}=G/K$.
Notice that for a finite-dimensional representation $\pi$, the fact that $D_{\pi}=G_{\C}$ follows directly from the definitions. Henceforth, all the representations that we consider will be infinite-dimensional. We also remark here that the conjecture is to be viewed as a complex-geometric description of the admissible dual of $G$. 

Now a Banach space $V$ is called \emph{uniformly convex} if given $\epsilon>0$ there exists a $\delta(\epsilon)>0$ such that $\frac{\norm{v+w}}{2}\leq 1-\delta(\epsilon)$ whenever $\norm{v}, \norm{w}\leq 1$ and $\norm{v-w}\geq \epsilon$. Further, a Banach space $V$
is called \emph{uniformly convex and uniformly smooth}, abbreviated henceforth as \emph{ucus}, if both $V$ and its continuous dual $V^{\ast}$ are uniformly convex (see \cite[Section $2$.a]{bfgm07}, for instance). We remark that Hilbert spaces as well as $\mathrm{L}^p$-spaces, $1<p<\infty,$ belong to this class. 

\hspace{-1em} Further, recall that a Banach representation $(\pi, V)$ of $G$ is said to be \emph{uniformly bounded} if there exists a constant $C>0$ such that 
\bq 
\label{eqn: uniform boundedness}
\norm{\pi(g)}_{op}\leq C \qquad \forall \, g\in G.
\eq
Here $\norm{\cdot}_{op}$ denotes the operator norm on the space of bounded linear operators on $V$. A Banach representation $\pi$ is called \emph{isometric} if $\norm{\pi(g)v}=\norm{v}$, $\forall \,g\in G$, $\forall \,v\in V$, where $\norm{\cdot}$ denotes the norm on $V$. Uniformly bounded Hilbert representations have been classically well-studied in the context of harmonic analysis on semisimple Lie groups (Kunze-Stein phenomenon, $(\g, K)$-module cohomology etc.). This is true also for representations on $\mathrm{L}^p$-spaces, while on general ucus Banach spaces, to our knowledge, such representations have been studied in the context of Property (T) and rigidity. 

\textbf{Remark}: \label{rem: equiv norm} Given a uniformly bounded Banach representation $(\pi, V)$, we can endow $V$ with the equivalent norm $\norm{v}_{\mathrm{isom}}:=\sup_{g\in G}\norm{\pi(g)v}$, $v\in V$, so that $\pi$ is isometric with respect to the new norm $\norm{\cdot}_{\mathrm{isom}}$.  
Of course, isometric representations are uniformly bounded representations for which the constant $C$ in \eqref{eqn: uniform boundedness} is $1$.

In this note, we show that for an infinite-dimensional
irreducible admissible uniformly bounded representation $\pi$ of a non-compact simple Lie group $G$, which is not of Hermitian type, on a ucus Banach space, the associated domain of holomorphy $D_{\pi}$ is the domain $\widetilde \Xi$ mentioned earlier in the section. We wish to point out that since a non-trivial irreducible, uniformly bounded Hilbert representation is necessarily admissible (see \cite[Theorem 5.2, Chapter IV]{borel-wallach}), our main result Theorem \ref{thm: non-Hermitian}, gives in particular the domains of holomorphy for all irreducible uniformly bounded Hilbert representations. We further emphasise that the class of irreducible uniformly bounded Hilbert representations is a much larger class than the class of irreducible unitary representations.
If $G$ is Hermitian and $\pi$ is a highest (respectively, lowest weight) representation of $G$, then we show that $D_{\pi}=\widetilde \Xi^+$, (respectively, $D_{\pi}=\widetilde \Xi^-$). It is not yet clear to us how to handle the case when $G$ is Hermitian but $\pi$ is neither a highest nor a lowest weight representation. In the unitary setting of \cite{kroetz08}, an $SL(2,\R)$-reduction is used, together with the uniqueness  of the direct integral decomposition for unitary representations. Since there is no analogous result on the uniqueness of direct integral decompositions in the Banach space setting, such a method does not extend.  
However, since the class of groups which are not Hermitian is a very large one, and since uniformly bounded representations are an important class, our results, while not complete, are, in our opinion, already of interest. 

Finally, our methods are an extension of the circle of ideas developed in \cite{kroetz-stanton04}, \cite{kroetz-opdam08}, and \cite{kroetz08}. An important point in the proof is that the holomorphic extension of the orbit map $\gamma_v$ of a non-zero $K$-finite vector $v$ depends essentially on the smooth Fr{\'e}chet structure of the Casselman-Wallach globalisation of the underlying $(\g, K)$-module. Another key observation is that from the vanishing at infinity of the matrix coefficients for non-trivial irreducible, uniformly bounded representations,  one can derive the appropriate properness of the $G$-action, which then leads to the determination of the desired domain of holomorphy. 

\section{Complex geometric setting}
\label{Section: complex geometry}
We begin by briefly describing the complex geometric setting, and refer to \cite{kroetz-stanton05}, \cite{kroetz-opdam08} for comprehensive accounts. With $G$ and $K$ as before, let $\g=\k\oplus \p$ be a Cartan decomposition such that $K$ is the analytic subgroup corresponding to $\k$, and let $\a$ be a maximal abelian subspace of $\p$. Set $\widehat{\Omega}:=\{Y \in \p \, |\, \text{spec}(\ad(Y))\subseteq (-\frac{\pi}{2}, \frac{\pi}{2})\}$, and $\Omega=\widehat{\Omega}\cap \a$, which is invariant under the Weyl group $W=W(\g,\a)$ associated to the set of restricted roots $\Sigma(\g,\a)$ of the pair $(\g,\a)$. Then we define the domain $\widetilde{\Xi}:=G\exp(i\Omega)K_{\C} \subset G_{\C}$ and, thereby, the so-called \emph{elliptic model} of the crown domain by $\Xi=\widetilde{\Xi}/K_{\C}$. It is known that $\Xi$ is a Stein domain admitting a proper $G$-action. Notice that $\mathbb{X}\subset \Xi\subset \mathbb{X}_{\C}$. Now, if we define the set of elliptic elements in $\mathbb X_{\C}$ by $\mathbb{X}_{\C,\;\textrm{ell}}:=G\exp(i\a)K_{\C}/K_{\C}$, then it is known that the crown domain $\Xi$ is the maximal domain contained in $\mathbb{X}_{\C,\;\textrm{ell}}$ which admits a proper $G$-action. However, $\Xi$ is \emph{not} a maximal domain in all of $\mathbb{X}_{\C}$ which admits a proper $G$-action. This is related to an alternative description of the crown domain - the so-called \emph{unipotent model}, and thence to the domains $\Xi^{\pm}$. Fixing an order on the restricted root system $\Sigma(\g,\a)$, let $\g=\k\oplus\a\oplus\n$  be the Iwasawa decomposition, and set $\Lambda$ to be the connected component of $\{Y\in \n\, |\, \exp(iY)K_{\C}/K_{\C}\in \Xi\}$ containing $0$. Then, as described in \cite[Section 8]{kroetz-opdam08}, we have the unipotent model $\Xi=G\exp(i\Lambda)K_{\C}/K_{\C}$ of the crown domain.  This model enables one to have a precise understanding of  the boundary of $\Xi$ which, in turn, then allows for a description of the directions in which the crown domain $\Xi$ can be extended to obtain a domain $D$ in such a way that it still admits a proper $G$-action. Denote by  
$\mathrm{\partial}\Xi$ the topological boundary of the crown domain $\Xi$. The \emph{elliptic} part 
$\mathrm{\partial}_{\textrm{ell}}\Xi$ of the boundary is then given by $\partial_{\textrm{ell}}\Xi=G\exp(i\mathrm{\partial}\, \Omega)K_{\C}$, and we define the \emph{unipotent} part of the boundary to be $\partial_u\Xi:=\partial\, \Xi\setminus \partial_{\textrm{ell}}\Xi$. Indeed, it can be seen that $\partial_u\Xi=G\exp(i\partial\Lambda)K_{\C}/K_{\C}$. Since the $G$-stabiliser of any point in $\partial_{\textrm{ell}}\Xi$ is a non-compact subgroup of $G$, it follows that for any $G$-invariant domain $D$  with $\mathbb X \subseteq D \subseteq \mathbb X_{\C}$ on which $G$ acts properly, we have  that $D\cap \partial_{\textrm{ell}}\Xi=\emptyset$. Further, $\partial_u\Xi\not\subset D$, and so if $D\not\subset \Xi$, then
$D \cap \partial_u\Xi\neq \emptyset$ (in fact, $D$ intersects the so-called \emph{regular} part of 
$\partial_u \Xi$).

For later use, we also mention that a simple real Lie group is called \emph{Hermitian} if the corresponding symmetric space $G/K$ admits the structure of a complex manifold. In this case, as a $\k_{\C}$-module, $\p_{\C}$ splits into irreducible components $\p_{\C}^+$, and $\p_{\C}^-$, and we let $\mathsf{P}^{\pm}$ denote the corresponding analytic subgroups of $G_{\C}$. Then it can be seen that $\widetilde \Xi^{\pm}=GK_{\C}\mathsf{P}^{\pm}$. 

\section{Holomorphic extensions of irreducible admissible representations}
\label{Section: holomorphic extensions}
In this section, we relate the complex geometric setting discussed above to irreducible admissible $G$-representations. The first important observation in this direction is the following result on the holomorphic extension of the orbit map of a non-zero $K$-finite vector to the $G\times K_{\C}$-invariant domain $\widetilde{\Xi}$.  We present the essential ideas of the proof, building on the proof of \cite[Theorem $3.1$]{kroetz-stanton04}. Note that admissibility is a crucial assumption in what follows. 

\begin{theorem}
\label{thm: crown domain}
Let $(\pi, V)$ be an irreducible, admissible Banach representation of a connected, non-compact, simple Lie group $G$. 
If $0\neq v \in V_K$ is a $K$-finite vector, 
then the orbit map $\gamma_v: G \longrightarrow V$ extends to a $G$-equivariant 
holomorphic map on $\widetilde{\Xi}=G\exp(i\Omega)K_{\C}$. 
\end{theorem} 
\begin{proof}

Let $V^\infty$ be the subspace of smooth vectors in $V$. Then $V_K\subset V^\infty$ and it is clear that $\gamma_v(G)\subset V^\infty$. 
Now, by the admissibility of $\pi$,  $V^\infty$ equipped with the topology induced by the universal enveloping algebra $\mathcal{U}(\g_{\C})$
becomes the (smooth) Casselman-Wallach  globalisation of the Harish-Chandra module $V_K$ (See \cite{casselman89}, \cite[Chapter $11$]{wallach2} or see \cite{bernstein-kroetz14} for a different approach).
On the other hand, since the topology on $V^\infty$ as a Casselman-Wallach globalisation 
is finer than the topology on it induced from $V$, we only need to prove that the orbit map  $\gamma_v: G \longrightarrow V^\infty$ extends to a $G$-equivariant  map on $\widetilde{\Xi}=G\exp(i\Omega)K_{\C}$ which is holomorphic with respect to the topology of
the Casselman-Wallach globalisation $(\pi^\infty, V^\infty)$. 

Now, by Casselman's submodule theorem, $(\pi^\infty, V^\infty)$ is embedded (as a closed $G$-submodule) 
into a smooth principal series representation 
$(\pi^{\infty}_{\tau, \lambda}, \mathcal H^{\infty}_{\tau,\lambda})$ 
arising from a minimal parabolic subgroup $P$ of $G$. 
In this way, we can assume that $v\in V^\infty\subseteq H^{\infty}_{\tau,\lambda}$. 
Then we can follow the argument in \cite[Theorem 3.1]{kroetz-stanton04} in order to conclude 
that $\gamma_v: G \longrightarrow V^\infty$ extends to a $G$-equivariant 
holomorphic map on $\widetilde{\Xi}$.
\end{proof}

This theorem tells us that the domains of holomorphy that we seek necessarily contain the domain $\widetilde{\Xi}$. The question then is whether such domains can be larger than \,$\widetilde \Xi$, and if so, to understand the connection between the geometry of the domains and representation theory. This crucial link is established by using the vanishing property of matrix coefficients at infinity of the irreducible, admissible representations under consideration, and relating this to properness of the $G$-action. We can then use the finer group theoretic structure according to whether $G$ is Hermitian or non-Hermitian to establish precisely what the sought-after domains of holomorphy are. Here is the result on vanishing at infinity of matrix coefficients that we need.

\begin{proposition}
\label{prop: vanishing at infty}
Let $(\pi, V)$ be an infinite-dimensional, uniformly bounded, irreducible, admissible representation of $G$ on a ucus Banach space. Then all the matrix coefficients of $\pi$ vanish at infinity.
\end{proposition}
\begin{proof}
For the case of uniformly bounded Hilbert representations, see \cite[Theorem $5.4$]{borel-wallach}). For isometric representations on general \emph{ucus} Banach spaces the result is due to Shalom and can be found, for instance, in \cite[Theorem $9.1$]{bfgm07}. By Remark \ref{rem: equiv norm}, the vanishing of matrix coefficients at infinity for uniformly bounded representations follows from the isometric case.
\end{proof}

We now use this proposition  to prove
the following theorem on the properness of the $G$-action.

\begin{theorem}
\label{thm: proper action}
Let $G$ be a non-compact, simple real Lie group and let  $(\pi, V)$ be an infinite-dimensional, uniformly bounded, irreducible  representation of $G$ on a ucus Banach space. Then $G$ acts properly on $D=\widetilde D/K_{\C}$, for any maximal $G\times K_{\C}$-invariant domain $\widetilde D\subset G_{\C}$ to which the orbit map $\gamma_v$ of a non-zero $K$-finite vector $v$ extends holomorphically. 
\end{theorem}
The first crucial step is the following lemma, which allows us to prove this theorem, not just for irreducible uniformly bounded Hilbert representations but also for irreducible admissible uniformly bounded representations on \emph{ucus} Banach spaces.  This extends \cite[Lemma $4.2$]{kroetz08} proved there in the unitary case.
\begin{lemma}
\label{lemma: proper action}
Let $(\pi, V)$ be an infinite-dimensional irreducible admissible uniformly bounded representation of $G$ on a ucus Banach space. Then $G$ acts properly on $V\setminus \{0\}$.
\end{lemma}

\begin{proof}
The $G$-action on $V\setminus \{0\}$ is proper if for every compact subset $W$ of $V\setminus \{0\}$, the set $W_G:=\{g\in G\,|\, \pi(g)W\cap W\neq \emptyset\}$ is also compact. Suppose $W_G$ is not compact for such a compact $W$. Then there exist sequences $(g_n)_{n\in\N}$ in $W_G$ and  $(v_n)_{n\in\N}$ in $W$ such that 
$\displaystyle{\lim_{n\to\infty} g_n=\infty}$ but
$\pi(g_n)v_n\in W$ for all $n\in \N$. By the compactness of $W$ we may assume, by going to subsequences if necessary, that there exist $v,v'\in W$ such that
$\displaystyle{\lim_{n\to\infty} v_n=v}$ and 
$\displaystyle{\lim_{n\to\infty} \pi(g_n)v_n=v'}$. Now, since $\pi$ is uniformly bounded, we have that 
$\norm{\pi(g_n)v_n-\pi(g_n)v}\leq C\norm{v_n-v}$, and so it follows that $\displaystyle{\lim_{n\to\infty} \pi(g_n)v=v'}$. Since $v'\neq 0$, there is an $f\in V^{\ast}$  such that $\langle v', f\rangle \neq 0$. But then  $\displaystyle{\lim_{n\to\infty}\langle\pi(g_n)v, f \rangle=\langle v', f\rangle \neq 0}$ - a contradiction to the vanishing of all the matrix coefficients at infinity, guaranteed by Proposition \ref{prop: vanishing at infty}. This concludes the proof. 
\end{proof}

Now we are in a position to prove Theorem \ref{thm: proper action}. This proof crucially uses the fact that the representation is uniformly bounded, thus extending \cite[Theorem 4.3]{kroetz-opdam08}.

\emph{Proof of Theorem 2}: 
We know that the orbit map $\gamma_v:G\to V$ extends to a $G$-equivariant holomorphic map $\widetilde D\to V$. Then using the identity theorem for holomorphic maps we can deduce, as in the proof of \cite[Theorem 4.3]{kroetz-opdam08}, that  $\pi(\tilde z)v\neq0$ for all $\tilde z \in \widetilde D$.

Now suppose that $G$ does not act properly on $D$. Then there exist sequences $(g_n)_{n\in\N}$ in $G$ with 
$g_n\to\infty$,  $(z_n)_{n\in\N}$ in $D$ such that $z_n\to z\in D$
and $g_nz_n\to w\in D$. We choose pre-images $\tilde z_n, \tilde z, \tilde w$ in $\widetilde D$ of $z_n, z, w$, respectively, such that $\tilde z_n\to \tilde z$. Then a sequence $(k_n)_{n\in\N}$ in $K_{\C}$ can also be chosen so that $g_n\tilde z_nk_n\to\tilde w$, from which it follows that $\pi(g_n\tilde z_n k_n)v\to \pi(\tilde w)v$. For a uniformly bounded representation, we note that sub-multiplicativity of the operator norm gives us also a bound from below:
\[
\frac{1}{C} \leq \norm{\pi(g)} \leq C \qquad \forall g\in G,
\]
where $C$ is as in \eqref{eqn: uniform boundedness}. So it follows that $\frac{1}{C}\norm{\pi(\tilde z_nk_n)v}\leq \norm{\pi(g_n\tilde z_nk_n)v}\leq C\norm{\pi(\tilde z_nk_n)v}$. Then since $\pi(\tilde z)v\neq0 \;\forall \tilde z \in \widetilde D$,  $\exists\; \alpha_1 ,\alpha_2>0$ such that $\alpha_1<\norm{\pi(\tilde z_nk_n)v}<\alpha_2$, for all $n\in\N$. Let $W$ denote the linear span of the $K$-translates of $v$. From these considerations and the fact that $\tilde z_n\to \tilde z$, we obtain that $\pi(\tilde z_n)_{|_{W}}-\pi(\tilde z)_{|_{W}}\to 0$, and hence that $\beta_1<\norm{\pi(k_n)v}<\beta_2$ for some constants $\beta_1, \beta_2>0$. As $v\in V_K$,  $W$ is finite-dimensional, and so the closure $U$ of the sequences $(\pi(g_n\tilde z_n k_n)v)_{n\in\N}$ and $(\pi(\tilde z_n k_n)v)_{n\in\N}$ in $V$ is a compact subset of $V\setminus \{0\}$. But then the set $U_G=\{g\in G\;|\; \pi(g)U\cap U\neq\emptyset\}$ contains the unbounded sequence $(g_n)_{n\in\N}$, and this is a contradiction to Lemma \ref{lemma: proper action}, thus proving the theorem. \qed

Now we are in a position to use the geometry of the situation according to if $G$ is Hermitian or not Hermitian, along with Theorem \ref{thm: crown domain} and Theorem \ref{thm: proper action} to determine the domains of holomorphy. First, suppose $G$ is non-Hermitian. Then as in \cite[Lemma 4.4]{kroetz08}, using an  $A_2$-reduction, we obtain that for any root $\alpha \in \Sigma$ and $Y \in \g_{\alpha}$, there exists an $m\in M=Z_K(\a)$ such that $\Ad(m)Y=-Y$. This, together with a certain precise description of the boundary of $\widetilde{\Xi}$ in the $SL(2,\R)$ case, 
leads, as in \cite[Theorem 4.1]{kroetz08}, to the geometric result that $D\subset \Xi$ for any $G$-invariant domain $D$ such that $\mathbb X \subset D \subset \mathbb X_{\C}$ on which $G$ acts properly. Therefore, we have that $\widetilde D/K_{\C}$ is contained in the crown domain $\Xi$. Now suppose $\widetilde D$ is a maximal $G\times K_{\C}$-invariant domain to which $\gamma_v$ extends holomorphically. Observe that by \autoref{thm: crown domain}, 
$\widetilde D$ necessarily contains the domain $\widetilde\Xi$.  \autoref{thm: proper action} then tells us that $G$ acts properly on $\widetilde D/K_{\C}$, and so it follows that it is equal to $\Xi$. This yields our main result in the non-Hermitian case.
 
\begin{theorem}
\label{thm: non-Hermitian}
Let $G$ be a connected real simple Lie group which is not Hermitian. If $(\pi, V)$ is  an infinite-dimensional irreducible admissible uniformly bounded representation of $G$ on a ucus Banach space, then the corresponding domain of holomorphy $D_{\pi}=\widetilde{\Xi}$. In particular, this is the case for all infinite-dimensional irreducible uniformly bounded Hilbert representations of $G$.
\end{theorem} 

We conclude by considering the case when $G$ is Hermitian. Here we use the notation set up in Section \ref{Section: complex geometry}. Suppose $(\pi, V)$ is  a highest (respectively, lowest) weight infinite-dimensional irreducible admissible uniformly bounded $G$-representation on a ucus Banach space $V$, i.e. $\mathsf{P}^{\pm}$ acts finitely on $V_K$. Using this, and Theorem \ref{thm: crown domain} as well as the fact that $V_K$ is also naturally a $K_{\C}$-module, we can conclude as in  \cite{kroetz08} that, for $0 \neq v \in V_K$, the orbit map $\gamma_v$ extends holomorphically to the $G\times K_{\C}$-invariant domain $GK_{\C}\mathsf{P}^{\pm}=\widetilde \Xi^{\pm}$. 
Now since $G$ is 
Hermitian, it is in fact true that if a $G$-invariant domain $D$ such that $\mathbb X \subset D \subset \mathbb X_{\C}$ admits a proper $G$-action, then either $D\subset \widetilde \Xi^+/K_{\C}$ or $D \subset \widetilde \Xi^-/K_{\C}$. Thus we can conclude using Theorem \ref{thm: proper action} that  
\begin{proposition}
For a Hermitian group $G$, we have that the domains of holomorphy $\widetilde D_{\pi}=\widetilde \Xi^{\pm}$ depending on whether $(\pi, V)$ is of highest weight or lowest weight, respectively.  
\end{proposition}


\bibliography{Bibliography}
\bibliographystyle{amsalpha}


\end{document}